\documentclass[reqno,centertags,a4paper,12pt]{amsart}
\usepackage{amssymb}
\usepackage{amsthm}
\usepackage{eucal}
\usepackage{color, xcolor}
\usepackage{graphicx}
\usepackage{amsmath}
\usepackage{amsfonts}
\usepackage{amsmath}

\usepackage[english]{babel}

\newcommand{\R}{\mathbb R}

\newcommand{\Z}{\mathbb Z}
\newcommand{\T}{\mathbb T}

\newtheorem{theorem}{Theorem}[section]

\newtheorem{lema}{Lemma}[section]

\newtheorem{prop}{Proposition}[section]
\newtheorem{rem}{Remark}[section]

\newcommand{\nor}[2]{{\left\|{#1}\right\|_{#2}}}
\newcommand{\nora}[3]{{\left\|{#1}\right\|_{#2}^{#3}}}

\title[On the well-posedness, ill-posedness and norm-inflation]{On the well-posedness, ill-posedness and norm-inflation for a higher order water wave model  on a periodic domain}

\author{X. Carvajal}
\address{Instituto de Matem\'atica, UFRJ, 21941-909, Rio de Janeiro, RJ, Brazil}
\email{carvajal@im.ufrj.br}

\author{M. Panthee}
\address{Department of Mathematics, IMECC-UNICAMP\\
13083-859, Campinas, S\~ao Paulo, SP,  Brazil}
\email{mpanthee@ime.unicamp.br}

\author{R. Pastr\'an}
\address{Universidad Nacional de Colombia, Bogot\'a, Colombia}
\email{rapastranr@unal.edu.co}

\thanks{MP was partially supported by FAPESP (2016/25864-6) Brazil CNPq (308131/2017-7) Brazil.}


\begin{document}
\maketitle
\begin{abstract}
In this work we are interested in the  well-posedness issues for the  initial value problem associated with a higher order water wave model posed on a pe\-rio\-dic domain $\T$. We derive some multilinear estimates and use them in the contraction mapping argument to prove local  well-posedness for initial data in the periodic Sobolev space $H^s(\mathbb{T})$, $s\geq 1$. With some restriction on the parameters appeared in the model, we use the conserved quantity to obtain global well-posedness for given data with Sobolev regularity $s\geq 2$.  Also, we use splitting argument to improve the global well-posedness result in $H^s(\mathbb{T})$ for $1\leq s< 2$.  Well-posedness result obtained in this work  is sharp   in the sense that the flow-map that takes initial data to the solution cannot to be continuous for given data in $H^s(\mathbb{T})$, $s< 1$.  Finally, we prove a norm-inflation result by showing that the solution corresponding to a smooth initial data may have arbitrarily large $H^s(\T)$ norm, with $s<1$, for arbitrarily short time.
\end{abstract}

\section{\sc Introduction}

Our interest in this work is to study the existence and other qualitative properties of the solution to the following  initial value problem (IVP) posed on a periodic domain $\T$
\begin{equation}\label{5kdvbbm}
\begin{cases}
\eta_t+\eta_x-\gamma_1 \eta_{xxt}+\gamma_2\eta_{xxx}+\delta_1 \eta_{xxxxt}+\delta_2\eta_{xxxxx}+\frac{3}{2}\eta \eta_x+\gamma (\eta^2)_{xxx}-\frac{7}{48}(\eta_x^2)_x-\frac{1}{8}(\eta^3)_x=0,   \\
\eta(x,0)=\eta_0(x),
\end{cases} 
\end{equation}
where
\begin{equation}\label{parametros}
\gamma_1=\frac{1}{2}(b+d-\rho),\qquad 
\gamma_2=\frac{1}{2}(a+c+\rho), 
\end{equation}
with $\rho = b+d-\frac16$, and 
\begin{equation}\label{parametros-1}
\begin{cases}
\delta_1=\frac{1}{4}\,[2(b_1+d_1)-(b-d+\rho)(\frac{1}{6}-a-d)-d(c-a+\rho)], \\
\delta_2 =\frac{1}{4}\,[2(a_1+c_1)-(c-a+\rho)(\frac{1}{6}-a)+\frac{1}{3}\rho], 
\gamma =\frac{1}{24}[5-9(b+d)+9\rho].
\end{cases}
\end{equation}
The parameters appeared in \eqref{parametros} and \eqref{parametros-1}  satisfy $a+b+c+d=\frac{1}{3}$, $\gamma_1+\gamma_2=\frac{1}{6}$, $\gamma=\frac{1}{24}(5-18\gamma_1)$ and $\delta_2-\delta_1=\frac{19}{360}-\frac{1}{6}\gamma_1$ with $\delta_1>0$ and $\gamma_1>0$.

The  higher order water wave model (\ref{5kdvbbm}) describes the unidirectional propagation of water waves was recently introduced  by Bona et al. \cite{BCPS1} by using the second order approximation in the two-way model, the so-called $abcd-$system introduced in \cite{BCS1, BCS2}. In the literature, this model  is also known as  the fifth order KdV-BBM type equation. The IVP \eqref{5kdvbbm} posed on the spatial domain $\R$ was studied by the authors in \cite{BCPS1} considering initial data in $H^s(\R)$ and proved local well-posedness for $s\geq 1$. When the parameter $\gamma$ satisfies $\gamma = \frac7{48}$, the model  \eqref{5kdvbbm}  posed on $\R$ possesses hamiltonian structure  and the flow satisfies
\begin{equation}\label{energy}
 E(\eta(\cdot,t)):=  \frac12\int_{\mathbb{R}} \eta^2 + \gamma_1 (\eta_x)^2+\delta_1(\eta_{xx})^2\, dx= E(\eta_0).
\end{equation} 
We note that, this conservation law holds in the periodic case as well (see \eqref{cq} below).

This energy conservation \eqref{energy} was used in \cite{BCPS1} to prove global well-posedness for $s\geq 2$ and low-high splitting technique to get global well-posedness for $s\geq\frac32$. This global well-posedness result was further improved in \cite{CP} for initial data with Sobolev regularity $s\geq 1$. Furthermore, the authors in \cite{CP} showed that the well-posedness result is sharp by proving that the mapping data-solution fails to be continuous at the origin whenever $s<1$.

As mentioned earlier, we are interested in studying the well-posedness issues for the IVP \eqref{5kdvbbm}  for given data in the periodic Sobolev space $H^s(\T)$. Ill-posedness issues are also considered.  To get well-posedness results we use idea from \cite{BT} and derive some multilinear estimates and use contraction mapping principle. To get ill-posedness results we show the failure of continuity of the flow-map at the origin. This sort of ill-posedness result was first introduced by Bourgain \cite{B} and further improved/refined by several authors, see for example \cite{BejT, BT, Mo, MST-1, MST-2, MV, MV-1} and references therein.

Before stating the main results, we record some notations that will be used throughout this this work. We use $C$ or $c$ to denote various space- and time-independent positive constants which may be different even in a single chain of inequalities. Given $a$, $b$ positive numbers, $a\lesssim b$ means that there exists a positive constant $C$ such that $a\leq C b$. We denote $a\sim b$ when, $a \lesssim  b$ and $b \lesssim a$. We will also denote $a\lesssim_{\lambda} b$ or $b\lesssim_{\lambda} a$, if the constant involved depends on some parameter $\lambda$. Given a Banach space $X$, we denote by $\nor{\cdot}{X}$ the norm in $X$. We will understand $\langle \cdot \rangle = (1+|\cdot|^2)^{1/2}$. $\mathcal{P}=C^{\infty}(\mathbb{T})$ denotes the space of all infinitely differentiable $2\pi$-periodic functions and $\mathcal{P}'$ will denote the space of periodic distributions, i.e., the topological dual of $\mathcal{P}$. For $f\in \mathcal{P}'$ we denote by $\widehat{f}$ the Fourier transform of $f$, $\widehat{f}=\left(\widehat{f}(k)\right)_{k\in\mathbb{Z}}$, where $$\widehat{f}(k)=\frac{1}{2\pi}\int_{0}^{2\pi}f(x)\,e^{-ikx}\,dx,$$ for all integer $k$. We will use the Sobolev spaces $H^s(\mathbb{T})$ equipped with the norm 
$$\nor{\phi}{H^s}=(2\pi)^{\frac{1}{2}} \nor{\langle k \rangle^{s}\widehat{\phi}(k)}{\ell^2(\mathbb{Z})}.$$ 
We will denote $\widehat{\eta}(k,t)$, $k\in\mathbb{Z}$, as the Fourier coefficient of $\eta$ respect to the variable $x$. 
Here are the main results. 

Now we state the first main result of this work that deals with the local well-posedness for given data in $H^s(\mathbb{T})$ for $s\geq 1$.
\begin{theorem}\label{lwp}
Let $\delta_1$ and $\gamma_1 >0$. Then, for any  given $\eta_0\in H^s(\mathbb{T})$, $s\geq 1$, there exists a time $T:=T(\|{\eta}_0\|_{H^s})=\frac{c_s}{\nor{\eta_0}{H^s}(1+\nor{\eta_0}{H^s})}$ and a unique solution  $\eta \in C([0,T];H^s(\mathbb{T}))$  of the IVP (\ref{5kdvbbm}) that depends continuously on the initial data. Moreover,  the correspondence $\eta_0 \mapsto \eta$ that associates to $\eta_0$ the solucion $\eta$ of the IVP \eqref{5kdvbbm} is a real analytic mapping of $B_R$ to $C([0,T], H^s(\mathbb{T})\,)$,  where for any $R>0$,  $B_R$ denotes the ball  in $H^s(\mathbb{T})$ centered at the origin with radius $R$   and  $T=T(R)>0$ denotes a uniform existence time for the IVP \eqref{5kdvbbm} with $\eta_0 \in B_R$.
\end{theorem}

As in the continuous case, with some restriction on the coefficients of the equation, we prove the following global well-posedness result in the periodic case too.
\begin{theorem}\label{gwp}
Assume $\delta_1$, $\gamma_1 >0$. Let $s\geq 1$ and $\gamma=\frac{7}{48}$. Then, the solution of the IVP (\ref{5kdvbbm}) given by Theorem \ref{lwp} can be extended to arbitrarily large time interval $[0,T]$. Hence, the IVP (\ref{5kdvbbm}) with $\gamma=\frac{7}{48}$ is globally well-posed in $H^s(\mathbb{T})$ for $s\geq 1$.
\end{theorem}

We also prove that the well-posedness results obtained in the previous theorems are sharp by showing that the application data-solution fails to be continuous at the origin whenever the data is given in $H^s(\T)$, $s<1$.  This is the content of the following theorem.
\begin{theorem}\label{ilwp-T}
Given $\eta_0\in H^s(\mathbb{T})$, $s<1$,  the IVP (\ref{5kdvbbm}) is ill-posed in the sense that  there exists $T>0$ such that  the flow-map $\eta_0 \mapsto \eta(t)$ for any $t\in (0,T)$, constructed in Theorem \ref{lwp} is discontinuous at the origin from $H^s(\mathbb{T})$ endowed with the topology inducted by $H^s(\mathbb{T})$ into $\mathcal{P}'(\mathbb{T})$.
\end{theorem}

Finally, we show the following result about norm-inflation which is in accordance to the third order BBM equation in \cite{BD}.

\begin{theorem}\label{norm-inflation}
For given any $s<1$, one can find a  sequence  of initial data $(\eta_0^j)_{j=1}^{\infty}\subset C^{\infty}$ such that
$$\eta_0^j\to 0, \quad {\textrm as}\quad j\to \infty $$ 
in $\dot{H}^s(\T)$ and a sequence $(T_j)_{j=1}^{\infty}$ of positive times with $T_j\to 0$ and $j\to\infty$ such that if $\eta_j(x,t)$ is the solution corresponding to the initial data $\eta_0^j$, then for all $j=1,2,\cdots$
$$\|\eta(\cdot, T_j)\|_{\dot{H}^s(\T)}\geq j.$$
\end{theorem}


\setcounter{equation}{0}
\section{\sc Local Theory in $H^s(\mathbb{T})$ for $s\geq 1$}

In this section, we address the local well-posedness for the IVP (\ref{5kdvbbm}), for given data in the periodic Sobolev spaces $H^s(\mathbb{T})$ with $s\geq 1$, using multilinear estimates combined with a contraction mapping argument. We write (\ref{5kdvbbm}) in an equivalent integral equation format. Taking the Fourier transform of equation in (\ref{5kdvbbm}) with respect to the spatial variable, we get, for all  integer $k$ 
\begin{equation}\label{ft5kdvbbm}
\widehat{\eta}_t+ik \widehat{\eta}+\gamma_1k^2 \widehat{\eta}_t-i\gamma_2k^3 \widehat{\eta}+\delta_1k^4 \widehat{\eta}_t+i \delta_2k^5 \widehat{\eta}+\frac{3}{4} ik\widehat{\eta^2}-i \gamma k^3 \widehat{\eta^2}-\frac{1}{8} ik \widehat{\eta^3}-\frac{7}{48} i k \widehat{\eta_x^2} =0.
\end{equation}
The equation \eqref{ft5kdvbbm}  can be simplified further to
\begin{equation}\label{ft5kdvbbm1}
i(1+\gamma_1k^2+\delta_1k^4)\widehat{\eta}_t= k (1 -\gamma_2k^2+ \delta_2k^4 )\widehat{\eta}+\frac{k}{4} (3 -4 \gamma k^2) \widehat{\eta^2}-\frac{1}{8} k \widehat{\eta^3}-\frac{7}{48} k \widehat{\eta_x^2} \; . 
\end{equation}
Since $\gamma_1$, $\delta_1$ are positive constants, the fourth-order polynomial 
\begin{equation}\label{symbol1}
\varphi(k):=1+\gamma_1k^2+\delta_1k^4,
\end{equation}
is strictly positive. Similarty to  \cite{BCPS1}, we define three Fourier multiplier operators $\phi(\partial_x)$, $\psi(\partial_x)$ and $\tau(\partial_x)$ via their symbols
\begin{equation}\label{symbols}
\widehat{\phi(\partial_x)f} (k):= \phi(k) \widehat{f}(k), \qquad \widehat{\psi(\partial_x)f} (k):= \psi(k) \widehat{f}(k) \quad \text{and}\quad \widehat{\tau(\partial_x)f} (k):= \tau(k) \widehat{f}(k),
\end{equation}
where
\begin{equation}\label{multipliers}
\phi(k):=\frac{k (1 -\gamma_2k^2+ \delta_2k^4 )}{\varphi(k)}\, , \qquad  \psi(k):= \frac{k}{\varphi(k)} \quad \text{and} \quad \tau(k):=\frac{k(3 -4 \gamma k^2)}{4\varphi(k)} \, .
\end{equation}
Then, in view of these definitions and \eqref{ft5kdvbbm1}, the IVP (\ref{5kdvbbm}) can be written in the following way
\begin{equation}\label{r5kdvbbm}
\left\{
\begin{aligned}
&i\eta_t=  \phi(\partial_x) \eta + \tau(\partial_x) \eta^2 -\frac{1}{8}\psi(\partial_x)\eta^3-\frac{7}{48} \psi(\partial_x)\eta_x^2,   \\
&\eta(x,0)=\eta_0(x).
\end{aligned}
\right. 
\end{equation}
Consider now the linear homogeneous IVP associated to (\ref{r5kdvbbm})
\begin{equation}\label{l5kdvbbm}
\left\{
\begin{aligned}
&i\eta_t=  \phi(\partial_x) \eta ,   \\
&\eta(x,0)=\eta_0(x).
\end{aligned}
\right. 
\end{equation}
Let $S$ be the unitary group in $H^s(\mathbb{T})$, $s\in \mathbb{R}$, generated by the operator $-i\phi(\partial_x)$. It is known that the solution of the IVP (\ref{l5kdvbbm}) is given by
\begin{equation}\label{unitgroup}
\eta(t)=S(t)\eta_0\, , \quad \text{where}\quad \widehat{S(t)\eta_0}(k)=e^{-i\phi(k)t}\widehat{\eta_0}(k)
\end{equation}
and, for all $t\geq 0$
\begin{equation}\label{unitgroup}
\nor{S(t)\eta_0}{H^s}=\nor{\eta_0}{H^s}\; .
\end{equation}
Using Duhamel's principle, the IVP (\ref{r5kdvbbm}) is equivalent to 
\begin{equation}\label{intequation-2}
\eta(x,t)=S(t)\eta_0-i\int_0^tS(t-t')\Bigl(\tau(\partial_x) \eta^2 -\frac{1}{8}\psi(\partial_x)\eta^3-\frac{7}{48} \psi(\partial_x)\eta_x^2 \Bigr)(x,t')\,dt'\, .
\end{equation}

In what follows, our objective is to solve the integral equation \eqref{intequation-2} using the  contraction mapping principle.
\subsection{\sc Multilinear Estimates}
In this subsection  we derive some multilinear estimates which will be useful in the proof of the local well-posedeness result. For motivation, we recall the ``sharp" bilinear estimate proved in \cite{BT} in the context of the third order BBM equation. Here, we will adapt ideas developed in \cite{BT} to address the fifth order model under consideration.
\begin{prop}\label{bilestbbm}
Let $u$, $v\in H^s(\mathbb{T})$ and $s\geq 0$. Then
\begin{equation}\label{bilest}
\nor{\omega (\partial_x)(u\,v)}{H^s}\lesssim_s \nor{u}{H^s} \nor{v}{H^s}\, ,
\end{equation}
where $\omega(\partial_x)$ is the Fourier multiplier operator defined by $\widehat{\omega(\partial_x) u} (k)=\omega(k)\widehat{u}(k)$ with
\begin{equation}\label{mult}
\omega(k)=\frac{|k|}{(1+k^2)}.
\end{equation}
The estimate (\ref{bilest}) is not valid for $s<0$.
\end{prop}
\begin{proof}
Expressing $\omega(\partial_x)uv$ in terms of Fourier transformed variables and using duality and a polarization argument, one may write (\ref{bilest}) in the equivalent form
\begin{equation}\label{rbilest}
\Bigl| \sum_{k\in \mathbb{Z}} \sum_{j\in \mathbb{Z}} \frac{|k| \langle k \rangle^s}{(1+k^2)\langle j \rangle^s \langle k-j\rangle^s}\,\widehat{a}(j)\widehat{b}(k-j)\overline{\widehat{c}(k)}\Bigr| \lesssim \nor{a}{L^2}\nor{b}{L^2}\nor{c}{L^2},
\end{equation}
where $\widehat{a}(k)=\langle k\rangle^s \widehat{u}(k)$, $\widehat{b}(k)=\langle k\rangle^s \widehat{v}(k)$ and $\overline{\widehat{c}(k)}=\langle k\rangle^{-s} \overline{\widehat{g}(k)}$. For $s\geq 0$, $\langle k\rangle^s \lesssim \langle j\rangle^s \langle k-j \rangle^{s}$. In consequence, the term  $\langle k \rangle^s / \langle j \rangle^s \langle k-j\rangle^s$ is bounded and may be ignored. Define $\widehat{c_1}(k)=\frac{|k|}{1+k^2}\, \overline{\widehat{c}(k)}$. With this notation, the left-hand side of (\ref{rbilest}) is simply $\langle \widehat{a}*\widehat{b}, \overline{\widehat{c_1}}\rangle$. Set $\widehat{a_1}(k)=\widehat{a}(-k)$ for all $k\in \mathbb{Z}$. Then, 
$$\langle \widehat{a}*\widehat{b}, \overline{\widehat{c_1}}\,\rangle=\langle \widehat{a_1}*\widehat{c_1}, \overline{\widehat{b}}\,\rangle ,$$
and using the Cauchy-Schwarz inequality and Young's inequality we have that
\begin{equation}\label{conv-11}
\begin{split}
\langle \widehat{a_1}*\widehat{c_1}, \overline{\widehat{b}}\,\rangle &\leq \nor{\widehat{a_1}*\widehat{c_1}}{\ell^2} \|{\widehat{b}}\|_{\ell^2} \leq \nor{\widehat{a_1}}{\ell^2} \nor{\widehat{c_1}}{\ell^1}\|{\widehat{b}}\|_{\ell^2} \\
&\leq \nor{\widehat{a}}{\ell^2}\nor{\frac{k}{1+k^2}}{\ell^2} \nor{\widehat{c}}{\ell^2}\|{\widehat{b}}\|_{\ell^2} \lesssim \nor{\widehat{a}}{\ell^2}\|{\widehat{b}}\|_{\ell^2} \nor{\widehat{c}}{\ell^2}.
\end{split}
\end{equation}
Using Plancherel's identity, the estimate \eqref{conv-11} finishes the proof of (\ref{bilest}). 

If $s<0$, let $\widehat{a}$, $\widehat{b}$ and $\widehat{c}$  be the characteristic functions of the subsets $\{N-1, N, N+1\}$,   $\{-N-1, -N, -N+1\}$ and  $\{-1, 0, 1\}$ respectively for some $N\in \mathbb{Z}$. Then the left-hand side of (\ref{rbilest}) behaves as $N^{-2s}$, while the right-hand side is a constant independent of $N$. Hence, if $s<0$, no matter how large is the implied constant, (\ref{bilest}) fails for $N\gg1$.
\end{proof}
\begin{prop}\label{multipliertau}
Let  $\tau(\partial_x)$ be the  operator as  defined in (\ref{multipliers}).
Then, for any $s\geq 0$, the following estimates hold
\begin{equation}\label{tau}
\nor{\tau(\partial_x)(\eta_1\eta_2)}{H^s}\lesssim_s \nor{\eta_1}{H^s} \nor{\eta_2}{H^s},
\end{equation}
\begin{equation}\label{dertau}
\nor{\partial_x\tau(\partial_x)(\eta_1\eta_2)}{H^1}\lesssim \nor{\eta_1}{H^1} \nor{\eta_2}{H^1}.
\end{equation}
\end{prop}
\begin{proof}
Since $\gamma_1$ and $\delta_1$ are positive constants, there exists a constant $C$ such that $\tau(k)\leq C \omega(k)$ for all $k\in \mathbb{Z}$. In view of this observation, the estimate (\ref{tau}) is consequence of the inequality (\ref{bilest}). In order to prove (\ref{dertau}), from definition of operator $\tau(\partial_x)$, we have that there exists a constant $C$ such that
$$|k \tau(k)|=\Bigl|\dfrac{k^2(3-4\gamma k^2)}{4(1+\gamma_1k^2+\delta_1k^4)}\Bigr|\leq C.$$
Hence, since $H^1$ is an algebra, we get
\begin{align*}
\nor{\partial_x\tau(\partial_x)(\eta_1\eta_2)}{H^1}&=(2\pi)^{1/2}\nor{\langle k\rangle k\, \tau(k) \widehat{(\eta_1\eta_2)}(k)}{\ell^2}\leq (2\pi)^{1/2} C \nor{\langle k\rangle \widehat{(\eta_1\eta_2)}(k)}{\ell^2} \\
&=C\nor{\eta_1\eta_2}{H^1}\lesssim \nor{\eta_1}{H^1} \nor{\eta_2}{H^1}.
\end{align*}
\end{proof}
\begin{prop}\label{multiplierpsi}
Let $\psi(\partial_x)$ be the  operator  as defined in (\ref{multipliers}). Then
for any $s>1/2$, the following estimates hold
\begin{equation}\label{psi}
\nor{\psi(\partial_x)(\eta_1\eta_2\eta_3)}{H^s}\lesssim_s \nor{\eta_1}{H^s} \nor{\eta_2}{H^s} \nor{\eta_3}{H^s},
\end{equation}
\begin{equation}\label{derpsi}
\nor{\partial_x\psi(\partial_x)(\eta_1\eta_2\eta_3)}{H^1}\lesssim \nor{\eta_1}{H^1} \nor{\eta_2}{H^1} \nor{\eta_3}{H^1}.
\end{equation}
\end{prop}
\begin{proof}
Observe that $|\psi(k)|\lesssim \omega(k)$ for all $k\in \mathbb{Z}$. Then, inequality (\ref{bilest}) and the fact that $H^s(\mathbb{T})$ is a Banach algebra for $s>1/2$ imply that 
\begin{align*}
\nor{\psi(\partial_x)(\eta_1\eta_2\eta_3)}{H^s}&=(2\pi)^{1/2}\nor{\psi(k)\widehat{(\eta_1\eta_2\eta_3)}(k)}{l^2}\lesssim \nor{\omega(\partial_x)(\eta_1\eta_2\eta_3)}{H^s}\\
&\lesssim_s \nor{\eta_1}{H^s}\nor{\eta_2\eta_3}{H^s}\lesssim_s \nor{\eta_1}{H^s}\nor{\eta_2}{H^s}\nor{\eta_3}{H^s}.
\end{align*}
To prove (\ref{derpsi}), from definition of operator $\psi(\partial_x)$, we have that there exists a constant $C$ such that
$$|k \psi(k)|=\Bigl|\dfrac{k^2}{1+\gamma_1k^2+\delta_1k^4}\Bigr|\leq C.$$

Exploring the property that $H^1$ is an algebra, one gets
\begin{align*}
\nor{\partial_x\psi(\partial_x)(\eta_1\eta_2\eta_3)}{H^1}&=(2\pi)^{1/2}\nor{\langle k\rangle k\, \psi(k) \widehat{(\eta_1\eta_2\eta_3)}(k)}{\ell^2}\leq (2\pi)^{1/2} C \nor{\langle k\rangle \widehat{(\eta_1\eta_2\eta_3)}(k)}{\ell^2} \\
&=C\nor{\eta_1\eta_2\eta_3}{H^1}\lesssim \nor{\eta_1}{H^1} \nor{\eta_2}{H^1} \nor{\eta_3}{H^1}.
\end{align*}
\end{proof}
\begin{prop}\label{multiplierpsietax}
Let $\psi(\partial_x)$ be the  operator  as defined in (\ref{multipliers}). Then for any $s\geq 1$, the following estimate hold
\begin{equation}\label{psix}
\nor{\psi(\partial_x)[(\eta_1)_x(\eta_2)_x]}{H^s}\lesssim_s \nor{\eta_1}{H^s} \nor{\eta_2}{H^s},
\end{equation}
\begin{equation}\label{derpsix}
\nor{\partial_x\psi(\partial_x)[(\eta_1)_x(\eta_2)_x]}{H^1}\lesssim \nor{\eta_1}{H^1} \nor{\eta_2}{H^1}.
\end{equation}
\end{prop}
\begin{proof}
Note that $\langle k \rangle \psi(k)\lesssim \omega(k)$ for all $k\in \mathbb{Z}$. Using Proposition \ref{bilestbbm} with $s-1\geq 0$ we have that
\begin{align*}
\nor{\psi(\partial_x)[(\eta_1)_x(\eta_2)_x]}{H^s}&=(2\pi)^{1/2}\nor{\langle k \rangle^s \psi(k)\widehat{[(\eta_1)_x(\eta_2)_x]}(k)}{\ell^2}\lesssim (2\pi)^{1/2}\nor{\langle k \rangle^{s-1} \omega(k)\widehat{[(\eta_1)_x(\eta_2)_x]}(k)}{\ell^2} \\
&=\nor{\omega(\partial_x)[(\eta_1)_x(\eta_2)_x]}{H^{s-1}}\lesssim_s \nor{(\eta_1)_x}{H^{s-1}}\nor{(\eta_2)_x}{H^{s-1}}\lesssim_s \nor{\eta_1}{H^s}\nor{\eta_2}{H^s}.
\end{align*}
To show (\ref{derpsix}), from definition of operator $\psi(\partial_x)$, we have that there exists a constant $C$ such that
$$|\langle k \rangle k\, \psi(k)|=\dfrac{\langle k \rangle k^2}{1+\gamma_1k^2+\delta_1k^4} \leq C\,\dfrac{|k|}{1+k^2}=C\,\omega(k).$$
Thus, using inequality (\ref{bilest}), we obtain 
\begin{align*}
\nor{\partial_x\psi(\partial_x)[(\eta_1)_x(\eta_2)_x]}{H^1}&=(2\pi)^{1/2}\nor{\langle k \rangle k \,\psi(k)\widehat{[(\eta_1)_x(\eta_2)_x]}(k)}{\ell^2}\lesssim (2\pi)^{1/2}\nor{\omega(k)\widehat{[(\eta_1)_x(\eta_2)_x]}(k)}{\ell^2} \\
&=\nor{\omega(\partial_x)[(\eta_1)_x(\eta_2)_x]}{L^2}\lesssim \nor{(\eta_1)_x}{L^2}\nor{(\eta_2)_x}{L^2}\lesssim \nor{\eta_1}{H^1}\nor{\eta_2}{H^1}.
\end{align*}
\end{proof}
\subsection{\sc Proof of local well-posedness}

In this subsection, we use the linear and nonlinear estimates derived above to prove the local well-posedness result stated in Theorem \ref{lwp}.

\begin{proof}[Proof of Theorem \ref{lwp}]
Let $\eta_0 \in H^s(\mathbb{T})$ and $s\geq 1$. We define the application
$$\Psi \eta (x,t)=S(t)\eta_0-i\int_0^tS(t-t')\Bigl(\tau(\partial_x) \eta^2 -\frac{1}{8} \psi(\partial_x)\eta^3-\frac{7}{48} \psi(\partial_x)\eta_x^2 \Bigr)(x,t')\,dt'\, ,$$
for each $\eta \in C([0,T];H^s(\mathbb{T}))$. As remarked before $S(t)$ is an unitary group in $H^s(\mathbb{T})$ (see (\ref{unitgroup})) and therefore
\begin{equation}\label{pf1}
\nor{\Psi \eta}{H^s}\leq \nor{\eta_0}{H^s}+CT\Bigl[ \nor{\tau(\partial_x) \eta^2 -1/8 \, \psi(\partial_x)\eta^3- 7/48 \, \psi(\partial_x)\eta_x^2}{C([0,T];H^s(\mathbb{T}))}\Bigr]\, .
\end{equation}
The inequalities (\ref{tau}), (\ref{psi}) and (\ref{psix}) imply that
\begin{equation}\label{pf2}
\nor{\Psi \eta}{H^s}\leq \nor{\eta_0}{H^s}+CT\Bigl[ \nora{\eta}{C([0,T];H^s(\mathbb{T}))}{2}+ \nora{\eta}{C([0,T];H^s(\mathbb{T}))}{3}\Bigr]\, .
\end{equation}
In the same way, for $\eta$ and $\mu$ in $C([0,T];H^s(\mathbb{T}))$ we have that
\begin{equation}\label{pf3}
\nor{\Psi \eta-\Psi \mu}{H^s}\leq CT \nor{\eta-\mu}{C([0,T];H^s(\mathbb{T}))} \Bigl[\nor{\eta+ \mu}{C([0,T];H^s(\mathbb{T}))}+\nor{\eta^2+ \eta \mu + \mu^2}{C([0,T];H^s(\mathbb{T}))}\Bigr]\, .
\end{equation}
So, taking the closed ball 
$$\mathcal{B}_r=\{u \in C([0,T];H^s(\mathbb{T}))\,:\,\nor{u}{C([0,T];H^s(\mathbb{T}))} \leq r \quad \text{with}\quad r=2\nor{\eta_0}{H^s}\}$$
we have that (\ref{pf2}) and (\ref{pf3}) become
\begin{align}
\nor{\Psi \eta}{H^s}&\leq \nor{\eta_0}{H^s}+CT(r^2+r^3)\, ,  \label{pf21} \\
\nor{\Psi \eta-\Psi \mu}{H^s}&\leq CT \nor{\eta-\mu}{C([0,T];H^s(\mathbb{T}))} (r +r^2)\, , \label{pf31}
\end{align}
where $\eta$, $\mu \in \mathcal{B}_r$. Hence, choosing $0<T\leq (2Cr(1+r))^{-1}$ we conclude that $\Psi \eta \in \mathcal{B}_r$ and $\Psi$ is a contraction on $\mathcal{B}_r$ . The rest of the proof follows a standard argument.
\end{proof}

\begin{rem}
From the proof of the Theorem \ref{lwp}, we can infer the following results.
\begin{itemize}
\item[(i)] The maximal existence time $T=T(\nor{\eta_0}{H^s})$ of the solution satisfies
\begin{equation}\label{existencetime}
\frac{1}{4C_s \nor{\eta_0}{H^s}(1+2\nor{\eta_0}{H^s})}=:\overline{T}\leq T,
\end{equation}
where the constant $C_s$ depends only on $s$.
\item[(ii)] The solution cannot grow too much, this means that
\begin{equation}\label{growsol}
\nor{\eta(\cdot, t)}{H^s}\leq r = 2 \nor{\eta_0}{H^s}
\end{equation}
for all $t\in [0,\overline{T}]$ where $\overline{T}$ is as above in (\ref{existencetime}).
\end{itemize}
\end{rem}


\setcounter{equation}{0}
\section{ Global Theory in $H^s(\mathbb{T})$ for $s\geq 1$}

In this section, we derive conserved quantity satisfied by the flow of \eqref{5kdvbbm} and  use it to obtain an \emph{a priori} estimate with an objective  to extend the local well-posedness result to the global in time. The present theory countenances the spaces $H^s(\mathbb{T})$, $s\geq 1$. However, we begin with a global well-posedness result in $H^s(\mathbb{T})$ for $s\geq 2$.

\subsection{\sc Global well-posedness in $H^s(\mathbb{T})$, $s\geq 2$}

Imposing  certain restrictions on the parameters that appear in \eqref{5kdvbbm},  we derive an \emph{a priori} estimate in $H^2(\mathbb{T})$. For this, multiply the equation in \eqref{5kdvbbm} by $\eta$, integrate by parts  over the spatial domain $[0, 2\pi]$ to obtain
\begin{equation}\label{multbyeta}
\dfrac{1}{2}\,\dfrac{d}{dt}\int_{0}^{2\pi}\bigl(\eta^2+\gamma_1\eta_x^2+\delta_1\eta_{xx}^2\bigr)\,dx + \gamma \int_{0}^{2\pi}\eta (\eta^2)_{xxx}\,dx - \dfrac{7}{48}\int_{0}^{2\pi}\eta (\eta_x^2)_x\,dx=0.
\end{equation}
Simplifying futher, one gets
\begin{equation}\label{ippmultbyeta}
\dfrac{1}{2}\,\dfrac{d}{dt}\int_{0}^{2\pi}\bigl(\eta^2+\gamma_1\eta_x^2+\delta_1\eta_{xx}^2\bigr)\,dx = \Bigl( \gamma - \dfrac{7}{48}\Bigr) \int_{0}^{2\pi}\eta_x^3\,dx.
\end{equation}

From \eqref{ippmultbyeta} it is clear that, if we consider $\gamma=\frac{7}{48}$, the quantity 
\begin{equation}\label{cq}
E(\eta(\cdot,t)):=  \frac12\int_{0}^{2\pi} \eta^2 + \gamma_1 (\eta_x)^2+\delta_1(\eta_{xx})^2\, dx,
\end{equation}
is conserved by the flow of \eqref{5kdvbbm}. Therefore, from now on, we consider $\gamma=\frac{7}{48}$ and use \eqref{cq} to obtain an \emph{a priori} estimate in $H^2(\T)$. Note that, with this consideration  the equation in \eqref{5kdvbbm} becomes
\begin{equation}\label{5kdvbbm748}
\eta_t+\eta_x-\gamma_1 \eta_{xxt}+\gamma_2\eta_{xxx}+\delta_1 \eta_{xxxxt}+\delta_2\eta_{xxxxx}+\frac{3}{2}\eta \eta_x+\gamma (\eta^2)_{xxx}-\gamma (\eta_x^2)_x-\frac{1}{8}(\eta^3)_x=0.
\end{equation}

As in the real line case \cite{CP}, using the conserved quantity \eqref{cq}, a standard argument implies the global well-posedness in $H^s(\mathbb{T})$, $s\geq 2$. More precisely, we have the following result.

\begin{theorem}
Let $s\geq 2$ and suppose $\gamma_1$, $\delta_1 >0$ and $\gamma=\frac{7}{48}$. Then the IVP \eqref{5kdvbbm} is globally well-posed in $H^s(\mathbb{T})$.
\end{theorem}

\subsection{\sc Global well-posedness in $H^s(\mathbb{T})$, $1\leq s< 2$}

In this subsection complete the proof the global result stated in  Theorem \ref{gwp}.

Let $\gamma=\frac{7}{48}$, $1\leq s< 2$ and $T>0$ be arbitrarily large but finite. Our aim in this part is to extend the local solution to the IVP \eqref{5kdvbbm} given by Theorem \ref{lwp} to the time interval $[0,T]$. For this purpose, we plan to use the splitting argument introduced in \cite{B1, B} and way earlier in \cite{BS}.

Consider $s\geq 1$, $N\gg 1$ to be chosen later and split the initial data $\eta_0\in H^s(\mathbb{T})$ in the low and high frequency part  $\eta_0=u_0+v_0$, with $\widehat{u_0}=\widehat{\eta_0}\chi_{\{|k|\leq N\}}$. One can easily show that $u_0\in H^{\delta}(\mathbb{T})$ for any $\delta \geq s$ and $v_0\in H^s(\mathbb{T})$ and satisfy the growth conditions
\begin{equation}\label{estuno}
\begin{aligned}
\nor{u_0}{L^2(\mathbb{T})}&\leq \nor{\eta_0}{L^2(\mathbb{T})}, \\
\nor{u_0}{\dot{H}^{\delta}(\mathbb{T})}&\leq \nor{\eta_0}{\dot{H}^s(\mathbb{T})} N^{\delta-s}, \qquad \delta\geq s,
\end{aligned}
\end{equation}
and
\begin{equation}\label{estdos}
\nor{v_0}{H^{\rho}(\mathbb{T})}\leq \nor{\eta_0}{H^s(\mathbb{T})} N^{\rho-s}, \qquad 0\leq \rho\leq s.
\end{equation}

Now,  the low frequency part $u_0$  of $\eta_0$ is evolved according to the IVP
\begin{equation}\label{ivpu}
\begin{cases}
i\,u_t =\phi(\partial_x)u + F(u), \\
u(x,0)=u_0(x),
\end{cases}
\end{equation}
where $F(u)=\tau(\partial_x)u^2-\frac{1}{8}\psi(\partial_x)u^3-\frac{7}{48}\psi(\partial_x)u_x^2$, and the high frequency part $v_0$ of $\eta_0$ according to the IVP
\begin{equation}\label{ivpv}
\begin{cases}
i\,v_t =\phi(\partial_x)v + F(v+u)-F(u), \\
v(x,0)=v_0(x).
\end{cases}
\end{equation}

Note that,  $\eta(x,t)=u(x,t)+v(x,t)$ solves the original IVP \eqref{5kdvbbm} in the time interval where both  $u$ and $v$ exist. In what follows, we prove the local  well-posedness of the IVP \eqref{ivpu} with existence time $ T_u$. Now, fixing the solution $u$ of the IVP  \eqref{ivpu}, we prove the local well-posedness of the  IVP \eqref{ivpv} with existence time $T_v$. Therefore, taking $t_0\leq \min\{T_u, T_v\}$, we see that $\eta= u+v$ is solution to the IVP \eqref{5kdvbbm} in  $[0, t_0]$ with data in $H^s(\mathbb{T})$, $s\geq 1$. Finally, we  iterate this process to cover any given time interval.

From Theorem \ref{lwp}  we see that the IVP \eqref{ivpu} is locally well-posed in $H^s(\mathbb{T})$, $s\geq 1$, with existence time given by $T_u=\dfrac{c_s}{\nor{u_0}{H^s}(1+\nor{u_0}{H^s})}$. In the following theorem we prove the local well-posedness of  the  variable coefficients IVP \eqref{ivpv}.

\begin{theorem}\label{lwpv}
Assume $\gamma_1$, $\delta_1 > 0$ and $u$ the solution of the IVP \eqref{ivpu}. For any given $v_0 \in H^s(\mathbb{T})$,  $s\geq 1$, there exists  $T_v=\dfrac{c_s}{(\nor{u_0}{H^s}+\nor{v_0}{H^s})(1+\nor{u_0}{H^s}+\nor{v_0}{H^s})}$ and a unique solution $v\in C([0,T_v]; H^s(\mathbb{T}))$  of the IVP \eqref{ivpv}. Moreover, the solution $v$ depends continuously on the initial data $v_0$.
\end{theorem} 
\begin{proof}
To prove this theorem, first we write the  IVP \eqref{ivpv} in its equivalent integral formulation
\begin{equation}\label{intequation}
\begin{aligned}
v(x,t)&=S(t)v_0 - i \int_0^tS(t-t')\bigl(F(v+u)-F(u)\bigr)(x,t')\,dt' \\
&:=S(t)v_0 + h(x,t)
\end{aligned}
\end{equation}
where
\begin{equation}\label{nonlinear}
F(v+u)-F(u)=\tau(\partial_x)\bigl(v^2+2uv\bigr)-\frac{1}{8}\psi(\partial_x)\bigl(v^3+3uv^2+3u^2v\bigr)-\frac{7}{48}\psi(\partial_x)\bigl(v_x^2+2v_xu_x\bigr).
\end{equation}
Let $u\in C([0,T_u]; H^s(\mathbb{T}))$ be the solution of the  IVP \eqref{ivpu} given by Theorem \ref{lwp} that satisfies
\begin{equation}\label{cotau}
\sup\limits_{t\in [0, T_u]}\nor{u(t)}{H^s(\mathbb{T})} \lesssim \nor{u_0}{H^s(\mathbb{T})}.
\end{equation}
Consider  $a:=2\nor{v_0}{H^s(\mathbb{T})}$ and define a ball
$$X_T^a=\{v\in C([0,T]; H^s(\mathbb{T})) \, : \, ||| v ||| := \sup\limits_{t\in [0, T]}\nor{v(t)}{H^s(\mathbb{T})}\leq a\}.$$
Now, we define an  application
$$\Phi_u(v)(x,t)=S(t)v_0 - i \int_0^tS(t-t')\bigl(F(v+u)-F(u)\bigr)(x,t')\,dt'.$$

As in the real line case (see \cite{CP}), using  the inequalities \eqref{tau}, \eqref{psi}, \eqref{psix} and \eqref{cotau}, for $T\leq T_u$, one can easily show that the application $\Phi_u$ is a contraction on $X_T^a$. The rest of the proof follows using  standard argument, so we omit the details.
\end{proof}

Now, we move to derive a crucial estimate to prove the global well-posedness result.
\begin{lema}\label{lema}
Let $s\geq 1$ and let $u$ be the solution of the IVP \eqref{ivpu} and $v$ the solution of the IVP \eqref{ivpv}. Then $h=h(u,v)$  defined in \eqref{intequation} is in $C([0,t_0]; H^2(\T))$. Moreover,
\begin{equation}\label{estimateuandh}
\nor{u(t_0)}{H^2(\T)}\lesssim N^{2-s} \qquad \text{and}\qquad \nor{h(t_0)}{H^2(\T)}\lesssim N^{s-3},
\end{equation}
where $t_0\sim N^{-2(2-s)}$.
\end{lema}
\begin{proof}
First, note that the energy conservation law \eqref{cq} implies
\begin{equation}\label{estimateu0}
\nor{u(t_0)}{H^2(\T)}\sim \sqrt{E(u(t_0))}=\sqrt{E(u_0)}\sim \nor{u_0}{H^2(\T)}\lesssim N^{2-s}.
\end{equation}
Using \eqref{intequation} and \eqref{nonlinear}, for $1\leq \delta \leq s$, one obtains
\begin{equation}\label{estimateh0}
\begin{aligned}
\nor{h(t_0)}{H^{\delta}(\T)}
&\leq \int_0^{t_0}\nor{F(v+u)-F(u)}{H^{\delta}(\T)}\,dt'\\
&\leq \int_0^{t_0}\Bigl(\nor{\tau(\partial_x)\bigl(v^2+2uv\bigr)}{H^{\delta}(\T)}+\frac{1}{8}\nor{\psi(\partial_x)\bigl(v^3+3uv^2+3u^2v\bigr)}{H^{\delta}(\T)} \\
&\qquad \qquad +\frac{7}{48} \nor{\psi(\partial_x)\bigl(v_x^2+2u_xv_x\bigr)}{H^{\delta}(\T)}\Bigr)\,dt' .
\end{aligned}
\end{equation}

Applying the estimates \eqref{tau}, \eqref{psi} and \eqref{psix}, one has
\begin{equation}\label{h0Hdelta}
\nor{h(t_0)}{H^{\delta}(\T)}\lesssim \int_0^{t_0}\nor{v}{H^{\delta}(\T)}\bigl(\nora{u}{H^{\delta}(\T)}{2}+ \nor{u}{H^{\delta}(\T)}+\nor{u}{H^{\delta}(\T)}\nor{v}{H^{\delta}(\T)}+\nor{v}{H^{\delta}(\T)}+\nora{v}{H^{\delta}(\T)}{2}\bigr)dt'.
\end{equation}

Note that, from local theory one has $\nor{u}{H^{\delta}(\T)}\lesssim \nor{u_0}{H^{\delta}(\T)}$ and $\nor{v}{H^{\delta}(\T)}\lesssim 2\nor{v_0}{H^{\delta}(\T)}$. Taking consideration of \eqref{estuno} and \eqref{estdos} one can conclude that $\nor{u}{H^{\delta}(\T)}\lesssim 1$ and $\nor{v}{H^{\delta}(\T)}\lesssim N^{\delta-s}$. Hence, taking $\delta=1$, we obtain for   $s\geq 1$, 
\begin{equation}\label{esth0}
\begin{aligned}
\nor{h(t_0)}{H^{1}(\T)}&\lesssim \int_0^{t_0}\Bigl(N^{1-s}+N^{2(1-s)}+N^{3(1-s)}\Bigr)dt' \\
&=t_0 \Bigl(N^{1-s}+N^{2(1-s)}+N^{3(1-s)}\Bigr) \\
&\sim N^{-2(2-s)} \Bigl(N^{1-s}+N^{2(1-s)}+N^{3(1-s)}\Bigr) \\
&\lesssim N^{s-3}.
\end{aligned}
\end{equation}

Now, using the estimates \eqref{dertau}, \eqref{derpsi} and \eqref{derpsix}, one gets
\begin{equation}\label{dxh0H1}
\nor{\partial_x h(t_0)}{H^{1}(\T)}\lesssim \int_0^{t_0}\nor{v}{H^{1}(\T)}\bigl(\nora{u}{H^{1}(\T)}{2}+ \nor{u}{H^{1}(\T)}+\nor{u}{H^{1}(\T)}\nor{v}{H^{1}(\T)}+\nor{v}{H^{1}(\T)}+\nora{v}{H^{1}(\T)}{2}\bigr)dt'.
\end{equation}
Similarly to \eqref{esth0}, one can easily get
\begin{equation}\label{estdxh0}
\nor{\partial_x h(t_0)}{H^1(\T)}\lesssim N^{s-3}.
\end{equation}

Finally, from  \eqref{esth0} and \eqref{estdxh0}, we obtain 
\begin{equation}\label{ht0H2}
\nor{h(t_0)}{H^2(\T)}\sim \nor{h(t_0)}{H^1(\T)}+\nor{\partial_x h(t_0)}{H^1(\T)}\lesssim N^{s-3},
\end{equation}
and this  completes the proof.
\end{proof}

Now we prove the following result that will  complete the proof of Theorem \ref{gwp}.
\begin{theorem}
Let $\gamma_1$, $\delta_1 >0$, $\gamma=\frac{7}{48}$ and $1\leq s <2$ be given. Then  the  IVP \eqref{5kdvbbm} is globally well-posed in the sense that, for any $T>0$, the solution given by Theorem \ref{lwp} can be extended to any time interval $[0, T]$. Moreover, the solution satisfies
\begin{equation}\label{taemH2}
\eta(t)-S(t)\eta_0 \in H^2(\T), \quad \text{for all time}\; t\in [0, T]
\end{equation}
and
\begin{equation}\label{estimaemH2}
\sup\limits_{t\in [0, T]}\nor{\eta(t)-S(t)\eta_0}{H^2(\T)}\lesssim (1+T)^{2-s},
\end{equation}
where $S(t)$ was defined in \eqref{unitgroup}.
\end{theorem}
\begin{proof}
Let $\eta_0 \in H^s(\T)$, $1\leq s < 2$ and $T>0$ be given. We decompose the given data $\eta_0 = u_0 + v_0$ in low and high frequency parts satisfying  the growth estimates \eqref{estuno} and \eqref{estdos}, respectively. 

The low frequency part  $u_0$ and the high frequency part $v_0$ are evolved according to the IVPs   \eqref{ivpu} and \eqref{ivpv} respectively.  Theorems \ref{lwp} and \ref{lwpv}, guarantee the existence of  the solutions $u$ and $v$. In this way the sum $\eta= u+v$ solves the IVP \eqref{5kdvbbm} in the common time interval of existence of $u$ and $v$.

Using \eqref{cq} and \eqref{estuno}, we find that
\begin{equation}\label{Eu0}
E(u(t))=E(u_0)\sim \nora{u_0}{H^2(\T)}{2}\lesssim N^{2(2-s)}.
\end{equation}
The local existence time in $H^2(\T)$, given by Theorem \ref{5kdvbbm}, can be estimated by
\begin{equation}\label{estTu}
\begin{aligned}
T_u &= \dfrac{c_s}{\nor{u_0}{H^2(\T)} \bigl(1+\nor{u_0}{H^2(\T)}\bigr) }  \\
&\gtrsim \dfrac{c_s}{N^{2-s} (1+N^{2-s})} \\
&\gtrsim \dfrac{c_s}{N^{2(2-s)}}:= t_0.
\end{aligned}
\end{equation}

Since $\bigl(\nor{u_0}{H^s(\T)}+\nor{v_0}{H^s(\T)}\bigr) \bigl(1+\nor{u_0}{H^s(\T)}+\nor{v_0}{H^s(\T)}\bigr)\lesssim \nor{\eta_0}{H^s(\T)}\bigl(1+\nor{\eta_0}{H^s(\T)}\bigr) = C_s$, one has
\begin{equation}\label{estTv}
T_v=\dfrac{c_s}{\bigl(\nor{u_0}{H^s(\T)}+\nor{v_0}{H^s(\T)}\bigr) \bigl(1+\nor{u_0}{H^s(\T)}+\nor{v_0}{H^s(\T)}\bigr)}\geq \dfrac{c_s}{C_s}\geq t_0.
\end{equation}

Using the estimates \eqref{estTu} and \eqref{estTv} one can infer  that the solutions $u$ and $v$ are both defined in the same time interval $[0, t_0]$. Also, from the  inequality \eqref{Eu0} we get the following bound on $t_0$
\begin{equation}\label{estt0}
t_0 \lesssim \dfrac{1}{E(u_0)}.
\end{equation}

 From  \eqref{intequation}, we see that  the solution $\eta$ at the time $t=t_0\sim N^{-2(2-s)}$, is given by
\begin{equation}\label{solutioneta}
\eta(t_0)=u(t_0)+v(t_0)=u(t_0)+S(t_0)v_0+h(t_0):= u_1 +v_1,
\end{equation} 
where
\begin{equation}\label{u1v1}
u_1:=u(t_0)+h(t_0) \quad \text{and} \quad v_1:=S(t_0)v_0.
\end{equation}

Now, at the time $t_0$ we consider the new initial data $u_1$ and $v_1$, and  evolve according to the IVP \eqref{ivpu} and \eqref{ivpv}, respectively, and continue iterating this process. In each step of  iteration we take the decomposition of the initial data as in \eqref{u1v1}. Hence $v_1$, $\cdots$, $v_k=S(kt_0)v_0$ have the  same $H^s(\T)$-norm as that of $v_0$, \emph{i.e.}, $\nor{v_k}{H^s(\T)}=\nor{v_0}{H^s(\T)}$. To continue with the iteration argument, we expect that $u_1$, $\cdots$, $u_k$  have  the same growth estimate as that satisfied by $u_0$. This will ensure the  existence time  in each iteration step has length $t_0$. In this way, we can extend the solution to any given time interval $[0, T]$. As in the real line case (see \cite{CP}), we complete this process using  induction argument. For the sake of clarity, we provide details considering $k=1$,  the other values  $k$ follows a similar argument. To accomplish this process we employ  the energy conservation \eqref{cq}. In fact, one has
\begin{equation}\label{Eu1}
E(u_1)=E(u(t_0)+h(t_0))=E(u(t_0))+\bigl(E(u_1)-E(u(t_0))\bigr).
\end{equation}
Then,  using Lemma  \ref{lema}, one can easily obtain
\begin{equation}\label{Eu1-Eut0}
\begin{aligned}
E(u_1)-E(u(t_0))&=\int_0^{2\pi}u(t_0)h(t_0)dx + \frac12\int_0^{2\pi}h(t_0)^2dx+\gamma_1\int_0^{2\pi}u_x(t_0)h_x(t_0)dx \\
&\quad +\frac{\gamma_1}{2}\int_0^{2\pi}h_x(t_0)^2dx+\delta_1\int_0^{2\pi}u_{xx}(t_0)h_{xx}(t_0)dx+\frac{\delta_1}{2}\int_0^{2\pi}h_{xx}(t_0)^2dx  \\
&\leq \nor{u(t_0)}{L^2(\T)}\nor{h(t_0)}{L^2(\T)}+ \frac12\nora{h(t_0)}{L^2(\T)}{2}+\gamma_1 \nor{u_x(t_0)}{L^2(\T)}\nor{h_x(t_0)}{L^2(\T)}\\
&\quad +\frac{\gamma_1}{2}\nora{h_x(t_0)}{L^2(\T)}{2}+\delta_1 \nor{u_{xx}(t_0)}{L^2(\T)}\nor{h_{xx}(t_0)}{L^2(\T)}+\frac{\delta_1}{2}\nora{h_{xx}(t_0)}{L^2(\T)}{2}\\
&\lesssim N^{2-s}N^{s-3}+\frac12N^{2(s-3)}+(\gamma_1+\delta_1)\bigl(N^{2-s}N^{s-3}+\frac12N^{2(s-3)}\bigr) \\
&\lesssim N^{-1}.
\end{aligned}
\end{equation}

Now, a combination of  \eqref{Eu1} and \eqref{Eu1-Eut0}, yields
\begin{equation}\label{estEu1}
E(u_1)\lesssim E(u(t_0))+ N^{-1}.
\end{equation}

Given any $T>0$, the number of steps required in the iteration process to cover the time interval $[0, T]$   is $\dfrac{T}{t_0}\sim TN^{2(2-s)}$. For this process to continue smoothly,  from \eqref{estEu1} it can be inferred that one needs to guarantee
$$TN^{2(2-s)}N^{-1}\lesssim N^{2(2-s)},$$
which holds  for $1\leq s <2$ and $N=N(T)=T$. 

Observe that, in each iteration we have $\nor{v_k}{H^2(\T)}=\nor{v_0}{H^2(\T)}$ and the  growth estimate
$$\nora{u_k}{H^2(\T)}{2}\sim E(u_k) \lesssim N^{2(2-s)}, \; \text{uniformly}.$$

Finally, for $t\in [0, T]$, there exists an integer  $k\geq 0$, with $t=kt_0+\tau$, for some $\tau \in [0,t_0]$. Therefore, in the $k^{th}$-iteration, one obtains
\begin{equation}\label{kiteration}
\begin{aligned}
\eta(t)&=u(\tau)+S(\tau)v_k +h(\tau) \\
&=u(\tau)+S(\tau)S(kt_0)v_0 +h(\tau) \\
&=u(\tau) + S(t)\eta_0-S(t)u_0+h(\tau).
\end{aligned}
\end{equation}

Therefore,
\begin{equation}
\eta(t)-S(t)\eta_0=u(\tau)-S(t)u_0+h(\tau),
\end{equation}
and this completes the proof.
\end{proof}



\setcounter{equation}{0}
\section{\sc Ill-posedness  in $H^s(\mathbb{T})$ for $s<1$}

In this section we consider the ill-posedness issue in the periodic case. The idea is similar to the one employed in \cite{CP} where the authors considered the continuous case with  a modification in the example of the initial data. 
\begin{proof}[Proof of Theorem \ref{ilwp-T}] 
Let $N, k_0 \in  \Z$ be fixed numbers such that $N \gg 1$, $1 \leq k_0 \sim 1$, $I_N= \{N-k_0, \dots, N, \dots, N+k_0\}$ and consider a sequence
\begin{equation}\label{defan}
a_k=\begin{cases} \dfrac{1}{N \sqrt{k_0} },\quad  &|k| \in I_N,\\
0,  \quad & otherwise.
\end{cases}
\end{equation}
Now we define $\eta_N$ via the Fourier transform by $\widehat{\eta_N}(k)= a_k$. An easy calculation shows that
\begin{equation}
\begin{split}
\|\eta_N \|_{H^1(\T)}^2= \sum_{k \in \Z} \langle k \rangle^2|\widehat{\eta_N}(k)  |^2=2 \sum_{j =-k_0}^{k_0} \langle N+j \rangle^2 \dfrac{1}{N^2 k_0} \sim 1,
\end{split}
\end{equation}
and 
\begin{equation}
\begin{split}
\|\eta_N \|_{H^s(\T)}^2= \sum_{k \in \Z} \langle k \rangle^{2s}|\widehat{\eta_N}(k)  |^2=2 \sum_{j =-k_0}^{k_0} \langle N+j \rangle^{2s} \dfrac{1}{N^2 k_0} \lesssim \dfrac{ 1}{N^{2(1-s)}} \to 0,
\end{split}
\end{equation}
for any $s<1$, when $N \to \infty$.

Recall that, while proving the local well-posedness result we used the Picard iteration scheme to find a unique fixed point that served as the solution to the IVP in question. In what follows, we show that the second iteration of this scheme
\begin{equation}\label{eq4.04}
I_2(h,h,x,t):=\int_0^tS(t-t')\Big(\tau(\partial_x)\big[S(t')h\big]^2
 -\frac7{48} \psi(\partial_x)\partial_x\big[S(t')h\big]^2\Big)  dt'
\end{equation}
 fails to be continuous  at the origin from $H^s(\T)$ to even $\mathcal{P}'(\T)$ for $s<1$.  

For this, we take $h=\eta_N$,  $0< t < T$ and compute the $H^s$ norm of $I_2(\eta_N,\eta_N,x,t)=:I_2$. Fourier transform in the spatial variable $x$, yields
\begin{equation}\label{eq4.05}
\begin{split}
\mathcal{F}_x  (I_2\,)(k) &=\int_0^t e^{-i(t-t')\phi(k)}\left(\tau(k)\big(\widehat{S(t')\eta_N\big)^2}(k)-\frac7{48}\psi(k)\widehat{\big(S(t')\partial_x\eta_N\big)^2}(k)\right)  dt'\\
&= \int_0^t e^{-i(t-t')\phi(k)}\left(\frac{3k-4\gamma k^3}{4\varphi(k)}\big(\widehat{S(t')\eta_N\big)^2}(k)-\frac7{48}\frac{k}{\varphi(k)}\widehat{\big(S(t')\partial_x\eta_N\big)^2}(k)\right)  dt'\\
&= \int_0^t e^{-i(t-t')\phi(k)}\left(\frac{3k-4\gamma k^3}{4\varphi(k)}\sum_{k_1\in\Z}e^{-it'\phi(k-k_1)}\widehat{\eta_N}(k-k_1)e^{-it'\phi(k_1)}\widehat{\eta_N}(k_1) \right.\\
&\qquad\qquad \left.-\frac7{48}\frac{k}{\varphi(k)}\sum_{k_1\in\Z}e^{-it'\phi(k-k_1)}(k-k_1)\widehat{\eta_N}(k-k_1)e^{-it'\phi(k_1)}k_1\widehat{\eta_N}(k_1)\right)  dt'\\
&=\sum_{k_1\in\Z} e^{-it\phi(k)}\left(\frac{3k-4\gamma k^3}{4\varphi(k)}-\frac7{48}\frac{kk_1(k-k_1)}{\varphi(k)}\right)\widehat{\eta_N}(k_1)\widehat{\eta_N}(k-k_1)   \int_0^te^{it'\Theta(k, k_1)} dt',
\end{split}
\end{equation}
where $\Theta(k, k_1):=\phi(k)-\phi(k-k_1)-\phi(k_1)$.
We have that
\begin{equation}\label{eq4.06}
 \int_0^te^{it[\Theta(k, k_1)]} dt'dk_1 = 
\frac{e^{it[\Theta(k, k_1)]}-1}{i[\Theta(k, k_1)]}.
\end{equation}

Now, inserting \eqref{eq4.06} in \eqref{eq4.05},  for any $k \in \Z$, we get
\begin{equation}\label{per1}
\mathcal{F}_x (I_2)(k)
=-i\sum_{k_1 \in \Z}\mathcal{X}(k, k_1)\widehat{\eta_N}(k-k_1) \widehat{\eta_N}(k_1)
 e^{-it\phi(k)}\frac{e^{it\Theta(k, k_1)}-1}{\Theta(k, k_1)},
\end{equation}
where  $\mathcal{X}(k, k_1):=\frac{k}{4\varphi(k)}\left(3-4\gamma k^2-\frac7{12}k_1(k-k_1)\right)$.

Let us define a set
\begin{equation}\label{def-Kper}
K_{k}:= \{ k_1 \in \Z: \quad k-k_1 \in I_N,\; -k_1\in I_N\}\cup \{ k_1 \in \Z:\quad k_1 \in I_N, \;-(k-k_1)\in I_N\}.
\end{equation} 
Thus, if $-2k_0\leq k \leq 2k_0$,
\begin{equation}\label{0per1}
\mathcal{F}_x (I_2)(k)
=-i\dfrac{1}{N^2 k_0}\sum_{k_1 \in K_{k}}\mathcal{X}(k, k_1)
 e^{-it\phi(k)}\frac{e^{it\Theta(k, k_1)}-1}{\Theta(k, k_1)}.
\end{equation}

With simple calculations, we can deduce that $|\mathcal{X}(k, k_1)| \sim N^2 k_0$ if $ k \neq 0$ and,
$
|\Theta(k, k_1)| \leq Ck_0
$. 
Now, considering $0<t< \dfrac{\pi}{4C k_0}$, we obtain
\begin{equation}\label{lb-tper}
\left|  \frac{e^{it\Theta(k, k_1)}-1}{\Theta(k, k_1)}\right|\geq  \dfrac{\sin(t\Theta(k, k_1) \,)}{t\Theta(k, k_1)}t\geq \cos(t\Theta(k, k_1)\,) t \geq  t\sqrt{2}/2.
\end{equation}

For simplicity we can suppose $k_0=1$. With this consideration we have  $I_N=\{N-1, N, N+1 \}$  and $K_1= \{-N, -N+1 \}$, and consequently
\begin{equation}\label{per2}
\begin{split}
\mathcal{F}_x (I_2)(1)
=&-i\dfrac{1}{N^2 }\mathcal{X}(1,- N)
 e^{-it\phi(1)}\frac{e^{it\Theta(1,- N)}-1}{\Theta(1, -N)}-i\dfrac{1}{N^2 }\mathcal{X}(1,- N+1)
 e^{-it\phi(1)}\frac{e^{it\Theta(1,- N+1)}-1}{\Theta(1, -N+1)}.
\end{split}
\end{equation}

Observe that $\mathcal{X}(1, -N+1)=\mathcal{X}(1, -N)+\frac{1}{4 \varphi(1)}(3-4\gamma -\frac76 N)$. So, it follows from \eqref{per2} that
\begin{equation}\label{per5}
\begin{split}
\mathcal{F}_x (I_2)(1)
=&-i\dfrac{1}{N^2 }\mathcal{X}(1,- N)  e^{-it\phi(1)} \left[  \frac{e^{it\Theta(1,- N)}-1}{\Theta(1, -N)}+\frac{e^{it\Theta(1,- N+1)}-1}{\Theta(1, -N+1)}  \right] \\
& -i\dfrac{e^{-it\phi(1)}}{4 \varphi(1)N^2 }
 (3-4\gamma -\frac76 N)\frac{e^{it\Theta(1,- N+1)}-1}{\Theta(1, -N+1)}\\
=:& \mathcal{P}_1+ \mathcal{P}_2.
\end{split}
\end{equation}

It is not difficult to see that
\begin{equation}\label{per6}
| \mathcal{P}_2| \lesssim  \dfrac{1}{N}.
\end{equation}

Now,  for $0<t< \dfrac{\pi}{4C k_0}$, from \eqref{lb-tper} we have
\begin{equation}\label{per7}
\begin{split}
| \mathcal{P}_1| \sim \dfrac{1}{N^2 }N^2 \left|  \frac{\sin\{it\Theta(1,- N)\}}{\Theta(1, -N)}+\frac{\sin\{it\Theta(1,- N+1)\}}{\Theta(1, -N+1)}  \right| \gtrsim t \sqrt2.
\end{split}
\end{equation}
Combining \eqref{per5}, \eqref{per6} and \eqref{per7}, we conclude that
$$
|\mathcal{F}_x (I_2)(1)| \gtrsim t>0.
$$
Hence
\begin{equation}\label{per3}
\begin{split}
\|I_2(\eta_N,\eta_N,t)\|_{H^s(\T)}^2=& \sum_{k \in \Z} \langle k  \rangle^{2s}|\mathcal{F}_x (I_2)(k)|^2
\gtrsim  t>0.
\end{split}
\end{equation}

By construction  $ \|\eta_N\|_{H^s(\T)} \to 0$ when $N \to \infty$ if $s<1$. This proves that for any fixed $t>0$, considering $\eta_0=\eta_N$, the map $\eta_0 \to  I_2(\eta_0,\eta_0,t)$ cannot be continuous at the origin from $H^s(\T)$ to even $\mathcal{P}'(\T)$.

Now, we will prove   that the discontinuity of the application $\eta_0 \to  I_2(\eta_0,\eta_0,t)$  implies the discontinuity of the flow-map $\eta_0\mapsto\eta(t)$ at the origin. Using, analyticity of the flow-map given by Theorem \ref{lwp},  we see that there exist $T>0$ and $\epsilon_0 >0$ such that for any $|\epsilon|\leq \epsilon_0$, any $\|h\|_{H^1(\T)}\leq 1$ and $0\leq t\leq T$, one has
\begin{equation}\label{eq2.7}
\eta(\epsilon h, t) = \epsilon S(t)h + \sum_{j=2}^{+\infty} \epsilon^j I_j(h^j, t),
\end{equation}
where $h^j:= (h, h, \cdots, h)$, $h^j\mapsto I_k(h^j, t)$ is a $j$-linear continuous map from $H^1(\T)^j$ into $C([0, T]; H^1(\T))$ and the series converges absolutely in $C([0, T]; H^1(\T))$.

From \eqref{eq2.7}, one obtains
\begin{equation}\label{eq2.8}
\eta(\epsilon \eta_N, t) -\epsilon^2I_2(\eta_N, \eta_N, t) = \epsilon S(t)\eta_N + \sum_{j=3}^{+\infty} \epsilon^j I_k(\eta_N^j, t).
\end{equation}

Also, we have that
\begin{equation}\label{eq2.9}
\|S(t)\eta_N\|_{H^s(\T)}\leq \|\eta_N\|_{H^s(\T)} \sim N^{s-1}
\end{equation}
and
\begin{equation}\label{eq2.10}
\Big\|\sum_{j=3}^{+\infty} \epsilon^j I_k(\eta_N^j, t)\Big\|_{H^1(\T)}\leq \Big(\frac{\epsilon}{\epsilon_0}\Big)^3 \sum_{J=3}^{+\infty} \epsilon_0^j \|I_j(\eta_N^j, t\|_{H^1(\T)}
\leq C\epsilon^3.
\end{equation}

Hence,   for any $s<1$, in view of \eqref{eq2.9} and \eqref{eq2.10} from \eqref{eq2.8}, we obtain
\begin{equation}\label{eq2.11}
\sup_{t\in [0, T]}\|\eta(\epsilon\eta_N, t) -\epsilon^2 I_2(\eta_N, \eta_N, t)\|_{H^s(\T)}\leq O(N^s) +C\epsilon^3.
\end{equation}

Let us fix $0<t<1$, take $\epsilon$ small enough and then $N$ large enough. Now, in view of  \eqref{eq2.7}, from  \eqref{eq2.11} we can conclude that for any $s<1$, $\epsilon^2 I_2(\eta_N, \eta_N, t)$ approximates $\eta(\epsilon\eta_N, t)$ in $H^s(\T)$ .

Choosing $0<\epsilon \ll 1$, from \eqref{eq2.8}, \eqref{eq2.9} and \eqref{eq2.10},  we get   
\begin{equation}\label{eq2.12}
\begin{split}
\|\eta(\epsilon\eta_N, t)\|_{H^s(\T)}&\geq \epsilon^2\|I_2(\eta_N, \eta_N, t\|_{H^s(\T)} -\epsilon\|S(t)\eta_N\|_{H^s(\T)}-\sum_{j=3}^{+\infty} \epsilon^j \|I_j(\eta_N^j, t\|_{H^s(\T)}\\
&\geq C_0\epsilon^2 -C_1\epsilon^3 - C\epsilon N^{s-1}\\
&\geq \frac{C_0}{2}\epsilon^2 -C\epsilon N^{s-1},
\end{split}
\end{equation}
for fixed $t>0$.

Now, we fix  $0<\epsilon \ll 1$  and choose $N$ sufficiently large, so that for any $s<1$, one gets from  \eqref{eq2.12}
\begin{equation}\label{eq2.13}
\|\eta(\epsilon\eta_N, t)\|_{H^s(\T)}\geq \frac{C_0}{4} \epsilon^2.
\end{equation}

Recall that, $\eta(0, t) \equiv 0$ and for any $s<1$, $\|\eta_N\|_{H^s(\T)}\to 0$. Hence, in the limit when $N\to \infty$  the flow-map $\eta_0\mapsto \eta(t)$ cannot be continuous  at the origin from $H^s(\T)$ to $C([0, 1]; H^s(\T)$, whenever $s<1$. Moreover, as $\eta_N \rightharpoonup 0$ in $H^1(\T)$, we also have that the flow-map is discontinuous from $H^1(\T)$ equipped with its weak topology inducted by  $H^s(\T)$ with values even in $\mathcal{P}'(\T)$.
\end{proof}

\setcounter{equation}{0}
\section{ Norm Inflation}
In this section we will prove the result on norm-inflation stated in Theorem \ref{norm-inflation}.  For technical reasons we make a change of variables $\eta(x,t) \equiv \eta(x-\dfrac{\delta_2}{\delta_1}t,t)$. With this change of variables the equation \eqref{5kdvbbm}  transforms to
\begin{equation}\label{x5kdvbbm}
\begin{cases}
\eta_t+\delta_3\eta_x-\gamma_1 \eta_{xxt}+\gamma_3\eta_{xxx}+\delta_1 \eta_{xxxxt}+\frac{3}{2}\eta \eta_x+\gamma (\eta^2)_{xxx}-\frac{7}{48}(\eta_x^2)_x-\frac{1}{8}(\eta^3)_x=0,\\
\eta(x, 0) = \eta_0.
\end{cases} 
\end{equation}
where $\delta_3=1-\dfrac{\delta_2}{\delta_1}>0$, $\gamma_3=\gamma_2+\gamma_1 \dfrac{\delta_2}{\delta_1}$.
This sort of change of variables introduced in \cite{BCG} eliminates the fifth order term  $\eta_{xxxxx}$ and only alters the coefficients of the terms $\eta_x$ and $\eta_{xxx}$. The formula \eqref{symbols} and \eqref{multipliers} remain the same with an exception of $\phi$ which in this case is replaced by
\begin{equation}\label{x15kdvbbm}
\tilde{\phi}(k)= \frac{k\left(\delta_3+\gamma_3 k^2\right)}{\varphi(k)},
\end{equation}
where $\varphi(k)=1+\gamma_1k^2+\delta_1k^4$ is the same symbol given in \eqref{symbol1}.  More precisely, the integral formulation in this case turns out to be
\begin{equation}\label{rintequation}
\eta(x,t)=S(t)\eta_0-i\int_0^tS(t-t')\Bigl(\tau(\partial_x) \eta^2 -\frac{1}{8}\psi(\partial_x)\eta^3-\frac{7}{48} \psi(\partial_x)\eta_x^2 \Bigr)(x,t')\,dt',
\end{equation}
where $S$ is the unitary group in $H^s(\mathbb{T})$, $s\in \mathbb{R}$, generated by the operator $-i\tilde{\phi}(\partial_x)$, i.e., $ \widehat{S(t)\eta_0}(k)=e^{-i\tilde{\phi}(k)t}\widehat{\eta_0}(k)$. 

Note that, being translation, the change of variables that we employed does not alter the $H^s$ norm in the new variables. Taking this point in consideration, we prove Theorem \ref{norm-inflation} for the equation \eqref{x5kdvbbm}.

\begin{proof}[Proof of Theorem \ref{norm-inflation}]
 Let $s<1$. The idea is to construct a sequence of initial data $\eta_0\in H^s$  that leads to the conclusion of the theorem. The second iteration of the Picard scheme applied on the integral formulation \eqref{rintequation} allows us to write the solution of the  IVP \eqref{x5kdvbbm} as
\begin{equation}\label{eta}
\eta(x,t)=S(t)\eta_0(x) + \eta_1(x,t)+\zeta(x,t),
\end{equation}
with
\begin{equation}\label{eq4.04}
\eta_1(x,t):=-i\int_0^tS(t-t')\Big[\tau(\partial_x)\big(S(t')\eta_0\big)^2
 -\frac7{48} \psi(\partial_x)\big(S(t')\eta_0\big)_x^2\Big]  dt'
\end{equation}
and
\begin{equation}\label{eq4.05}
\zeta(x,t):=-i\int_0^tS(t-t')\Bigl[\tau(\partial_x)F_1(t')-\frac1{8} \psi(\partial_x) F_2(t')
 -\frac7{48} \psi(\partial_x) F_3(t')\Bigr]  dt',
\end{equation}
where
\begin{align*}
F_1(t')&=F_{11}(t')+F_{12}(t')+F_{13}(t'), \\
F_{11}(t')&=\eta_1^2(t')+2\eta_1(t')\big(S(t')\eta_0\big), \\
F_{12}(t')&=2\zeta(t')\eta_1(t')+2\zeta(t')S(t')\eta_0, \\
F_{13}(t')&=\zeta^2(t'),
\end{align*}
\begin{align*}
F_2(t')&=F_{21}(t')+F_{22}(t')+F_{23}(t')+F_{24}(t'), \notag \\
F_{21}(t')&=\eta_1^3(t')+3\eta_1^2(t') S(t')\eta_0+3\eta_1(t')\big(S(t')\eta_0\big)^2+\big(S(t')\eta_0\big)^3, \notag \\
F_{22}(t')&=3\zeta(t')\eta_1^2(t')+6\zeta(t')\eta_1(t')S(t')\eta_0+3\zeta(t')\big(S(t')\eta_0\big)^2, \notag \\
F_{23}(t')&=3\zeta^2(t')S(t')\eta_0 + 3\zeta^2(t')\eta_1(t'), \notag \\
F_{24}(t')&=\zeta^3(t'),
\end{align*}
and
\begin{align*}
F_3(t')&=F_{31}(t')+F_{32}(t')+F_{33}(t'), \\
F_{31}(t')&=(\eta_1)_x^2(t')+2(\eta_1)_x(t')\big(S(t')\eta_0\big)_x, \\
F_{32}(t')&=2(\eta_1)_x(t')\zeta_x(t')+2\zeta_x(t')\big(S(t')\eta_0\big)_x, \\
F_{33}(t')&=\zeta_x^2(t').
\end{align*}

Recall that $ \widehat{S(t)\eta_0}(k)=e^{-i\tilde{\phi}(k)t}\widehat{\eta_0}(k)$ with $\tilde{\phi}$ defined in \eqref{x15kdvbbm}. For simplicity of exposition, in what follows, we delete tilde sign and use $\phi$ in place of $\tilde{\phi}$.

 The linear operator $S(t)$ translates the wave and preserves its magnitude, i.e., for all $k=1, 2, 3, \cdots$,
\begin{equation}\label{op-1}
S(t)\sin(kx) =\sin(kx-t\phi(k)), \quad\qquad S(t)\cos(kx) =\cos(kx-t\phi(k)).
\end{equation}
However,  the operators $\phi(\partial_x)$, $\tau(\partial_x)$ and $\psi(\partial_x)$ change the amplitude, add rotation and vanish on constant functions
\begin{equation}\label{op-3}
\phi(\partial_x)\sin(kx-\ell t) =-i\phi(k)\cos(kx-\ell t), \qquad \phi(\partial_x)\cos(kx-\ell t) =i\phi(k)\sin(kx-\ell t),
\end{equation}
\begin{equation}\label{op-tau}
\tau(\partial_x)\sin(kx-\ell t) =-i\tau(k)\cos(kx-\ell t), \qquad \tau(\partial_x)\cos(kx-\ell t) =i\tau(k)\sin(kx-\ell t),
\end{equation}
\begin{equation}\label{op-psi}
\psi(\partial_x)\sin(kx-\ell t) =-i\psi(k)\cos(kx-\ell t), \qquad \psi(\partial_x)\cos(kx-\ell t) =i\psi(k)\sin(kx-\ell t).
\end{equation}

Now, we move to construct initial data announced in the beginning of the proof. For $k_1, k_2\in\mathbb{Z}$ large with $k_2=k_1+1$, choose
\begin{equation}\label{eta_bar}
\bar\eta=\sin(k_1x)+\sin(k_2x)
\end{equation}
and a mean zero initial data
\begin{equation}\label{eta_0}
\eta_0=k_1^{\sigma -1}\bar\eta, \quad 0<\sigma<1-s.
\end{equation}

As in the case of the third order BBM equation (see \cite{BD}), for  $\eta_0$ chosen in \eqref{eta_0}, we will prove that whenever $s<1$ the $H^s(\T)$-norm of  corresponding $\eta_1$ in \eqref{eta}  becomes large in a short time while the error term $\zeta$  stays bounded in the same space. 

Note that
\begin{equation}\label{norm-i1}
\|S(t)\eta_0\|_{H^s} = \|\eta_0\|_{H^s} = \Big(\sum_{k\in\Z}\langle k\rangle^{2s}|\widehat{\eta_0}(k)|^2\Big)^{\frac12} \sim k_1^{\sigma-1+s}.
\end{equation}
The initial data $\eta_0$ has small $H^s(\T)$-norm as $\sigma -1+s<0$.
Now, 
\begin{equation}\label{norm-i2} 
S(t')\bar\eta = \sin(k_1x-t'\phi(k_1)) +\sin(k_2x-t'\phi(k_2))
\end{equation}
 and
\begin{equation} \label{norm-i3}
\begin{split}
\big[S(t')\bar\eta\big]^2 &= \frac12\big[1-\cos(2k_1x-2t'\phi(k_1))\big]+ \frac12\big[1-\cos(2k_2x-2t'\phi(k_2))\big]\\
&\quad +\cos[(k_1-k_2)x-t'(\phi(k_1)-\phi(k_2))]-\cos[(k_1+k_2)x-t'(\phi(k_1)+\phi(k_2))].
\end{split}
\end{equation}
Using \eqref{op-tau} one can get
\begin{equation} \label{norm-i4}
 \begin{split}
\tau(\partial_x)\big[S(t')\bar\eta\big]^2 &= -\frac{i}2\tau(k_1)\sin(2k_1x-2t'\phi(k_1))- \frac{i}2\tau(k_2)\sin(2k_2x-2t'\phi(k_2))\\
&\quad +i\tau(k_1-k_2)\sin[(k_1-k_2)x-t'(\phi(k_1)-\phi(k_2))]\\
&\quad - i\tau(k_1+k_2)\sin[(k_1+k_2)x-t'(\phi(k_1)+\phi(k_2))]\\
&=:I_1+I_2+I_3+I_4.
\end{split}
\end{equation}

Differentiating \eqref{norm-i3} with respect to $x$, we obtain that
\begin{equation}\label{norm-i5} 
\partial_xS(t')\bar\eta = k_1\cos(k_1x-t'\phi(k_1)) +k_2\cos(k_2x-t'\phi(k_2)),
\end{equation}
and
\begin{equation}\label{norm-i6} 
\begin{split}
\big[\partial_xS(t')\bar\eta\big]^2& = \frac{k_1^2}2\Big[1+\cos(2k_1x-2t'\phi(k_1) \Big]+\frac{k_2^2}2\Big[1+\cos(2k_2x-2t'\phi(k_2)) \Big]\\
 &\quad +k_1k_2\cos[(k_1-k_2)x-t'(\phi(k_1)-\phi(k_2))]\\
 &\quad +k_1k_2\cos[(k_1+k_2)x-t'(\phi(k_1)+\phi(k_2))].
 \end{split}
\end{equation}
Using \eqref{op-psi} we obtain that 
\begin{equation} \label{norm-i7}
\begin{split}
-\frac7{48}\psi(\partial_x)\big[\partial_xS(t')\bar\eta\big]^2& = \frac{-7ik_1^2}{96}\psi(k_1)\sin(2k_1x-2t'\phi(k_1) \Big]+\frac{-7ik_2^2}{96}\psi(k_2)\sin(2k_2x-2t'\phi(k_2)) \\
 &\quad -\frac{7ik_1k_2}{48}\psi(k_1-k_2)\sin[(k_1-k_2)x-t'(\phi(k_1)-\phi(k_2))] \\
 &\quad -\frac{7ik_1k_2}{48}\psi(k_1+k_2)\sin[(k_1+k_2)x-t'(\phi(k_1)+\phi(k_2))]\\
 &=:J_1+J_2+J_3+J_4.
 \end{split}
\end{equation}
With what we did above,
\begin{equation}\label{int-n1}
\begin{split}
\eta_1&=k_1^{2(\sigma -1)}\sum_{j=1}^{4}\int_0^tS(t-t')[I_j + J_j]dt'\\
& := k_1^{2(\sigma -1)}\sum_{j=1}^{4}\mathcal{I}_j+k_1^{2(\sigma -1)} \sum_{j=1}^{4}\mathcal{J}_j ,
\end{split}
\end{equation}
where
$$
\mathcal{I}_j:=\int_0^tS(t-t') I_j\,dt', \quad \textrm{and} \quad \mathcal{J}_j:=\int_0^tS(t-t')J_j\,dt',
$$
for $j=1, \dots,4$.

 A simple calculation yields
\begin{equation}\label{norm-i8}
\begin{split}
\int_0^tS(t-t')\sin(kx-\ell t')dt'& = \int_0^t\sin\big(kx-\phi(k)(t-t')-\ell t'\big)dt'\\
&=\Big(\phi(k)-\ell\Big)^{-1}\Big(\cos\big(kx-\phi(k)t\big)-\cos\big(kx-\ell t\big)\Big).
\end{split}
\end{equation}
In view of \eqref{norm-i4} and \eqref{norm-i8}, we have that
\begin{equation}\label{int-n2}
\begin{split}
\mathcal{I}_1& =-\frac{i}2 \tau(k_1)\int_0^tS(t-t')\sin(2k_1x-2\phi(k_1)t')dt'\\
&=-\frac{i}2 \tau(k_1)\big[\phi(2k_1)-2\phi(k_1)\big]^{-1}\big[\cos\big(2k_1x-\phi(2k_1)t\big)-\cos\big(2k_1x-2\phi(k_1)t\big)\big].
\end{split}
\end{equation}
 Similarly,
 \begin{equation}\label{int-n3}
\begin{split}
\mathcal{I}_2& =-\frac{i}2 \tau(k_2)\int_0^tS(t-t')\sin(2k_2x-2\phi(k_2)t')dt'\\
&=-\frac{i}2 \tau(k_2)\big[\phi(2k_2)-2\phi(k_2)\big]^{-1}\big[\cos\big(2k_2x-\phi(2k_2)t\big)-\cos\big(2k_2x-2\phi(k_2)t\big)\big],
\end{split}
\end{equation}
\begin{equation}\label{int-n4}
\begin{split}
\mathcal{I}_3& =i \tau(k_1-k_2)\int_0^tS(t-t')\sin((k_1-k_2)x-(\phi(k_1)-\phi(k_2)t')dt'\\
&=i\tau(k_1-k_2)\big[\phi(k_1-k_2)-\phi(k_1)+\phi(k_2)\big]^{-1}\times\\
&\qquad \big[\cos\big((k_1-k_2)x-\phi(k_1-k_2)t\big)-\cos\big((k_1-k_2)x-(\phi(k_1)-\phi(k_2))t\big)\big],
\end{split}
\end{equation}
and
\begin{equation}\label{int-n5}
\begin{split}
\mathcal{I}_4& =-i \tau(k_1+k_2)\int_0^tS(t-t')\sin((k_1+k_2)x-(\phi(k_1)+\phi(k_2)t')dt'\\
&=-i\tau(k_1+k_2)\big[\phi(k_1+k_2)-\phi(k_1)-\phi(k_2)\big]^{-1}\times\\
&\qquad \big[\cos\big((k_1+k_2)x-\phi(k_1+k_2)t\big)-\cos\big((k_1+k_2)x-(\phi(k_1)+\phi(k_2))t\big)\big].
\end{split}
\end{equation}
Also, from \eqref{norm-i7} and \eqref{norm-i8}, we obtain that
\begin{equation}\label{int-n6}
\begin{split}
\mathcal{J}_1& =\frac{ik_1^2}2 \psi(k_1)\int_0^tS(t-t')\sin(2k_1x-2\phi(k_1)t')dt'\\
&=\frac{ik_1^2}2 \psi(k_1)\big[\phi(2k_1)-2\phi(k_1)\big]^{-1}\big[\cos\big(2k_1x-\phi(2k_1)t\big)-\cos\big(2k_1x-2\phi(k_1)t\big)\big].
\end{split}
\end{equation}
Similarly,
 \begin{equation}\label{int-n7}
\begin{split}
\mathcal{J}_2&=\frac{ik_2^2}2 \psi(k_2)\int_0^tS(t-t')\sin(2k_2x-2\phi(k_2)t')dt'\\
&=\frac{ik_2^2}2 \psi(k_2)\big[\phi(2k_2)-2\phi(k_2)\big]^{-1}\big[\cos\big(2k_2x-\phi(2k_2)t\big)-\cos\big(2k_2x-2\phi(k_2)t\big)\big],
\end{split}
\end{equation}
\begin{equation}\label{int-n8}
\begin{split}
\mathcal{J}_3& =i k_1k_2\psi(k_1-k_2)\int_0^tS(t-t')\sin((k_1-k_2)x-(\phi(k_1)-\phi(k_2)t')dt'\\
&=i k_1k_2\psi(k_1-k_2)\big[\phi(k_1-k_2)-\phi(k_1)+\phi(k_2)\big]^{-1}\times\\
&\qquad \big[\cos\big((k_1-k_2)x-\phi(k_1-k_2)t\big)-\cos\big((k_1-k_2)x-(\phi(k_1)-\phi(k_2))t\big)\big]
\end{split}
\end{equation}
and
\begin{equation}\label{int-n9}
\begin{split}
\mathcal{J}_4& =i k_1k_2\psi(k_1+k_2)\int_0^tS(t-t')\sin((k_1+k_2)x-(\phi(k_1)+\phi(k_2)t')dt'\\
&=i k_1k_2\psi(k_1+k_2)\big[\phi(k_1+k_2)-\phi(k_1)-\phi(k_2)\big]^{-1}\times\\
&\qquad \big[\cos\big((k_1+k_2)x-\phi(k_1+k_2)t\big)-\cos\big((k_1+k_2)x-(\phi(k_1)+\phi(k_2))t\big)\big].
\end{split}
\end{equation}
 We will estimate only three terms: $\mathcal{I}_4$, $\mathcal{J}_1$ and $\mathcal{J}_3$. The estimate to the other terms is very similar. Using 
\begin{equation}
\cos(A)- \cos(B) =2\sin \left(\frac{A+B}2\right)\sin \left(\frac{B-A}2\right),
\end{equation}
 for $k_1$ large, we have
\begin{equation}\label{x1int-n8}
\begin{split}
\int_0^tS(t-t')I_4dt' 
&=-i\tau(k_1+k_2)\dfrac{t \sin (t I_{4 1})}{tI_{4 1}}\sin \left( (k_1+k_2)x+I_{4 2} t  \right)\\
&\sim -i\tau(k_1+k_2) t \sin \left( (k_1+k_2)x +I_{4 2} t \right),
\end{split}
\end{equation}
where $I_{4 1}=\phi(k_1+k_2)-\phi(k_1)-\phi(k_2) $  and $I_{4 2}=\phi(k_1)+\phi(k_2)+\phi(k_1+k_2)$ are very small if $k_1$ is large. Observe that   $\dfrac{\sin (t I_{4 1})}{tI_{4 1}} \sim 1$.
Similarly,  for $k_1$ large
\begin{equation}\label{x1int-n8}
\begin{split}
\int_0^tS(t-t')J_1dt' 
=&\frac{ik_1^2}2 \psi(k_1)\, \dfrac{ t \sin t J_{1 1}}{tJ_{1 1}}\sin \left( 4k_1 x+J_{1 2} t  \right)\\
\sim &\frac{ik_1^2}2 \psi(k_1)\, t \sin \left( 4k_1 x  +J_{1 2} t\right),
\end{split}
\end{equation}
where $J_{1 1}=\phi(2k_1)-2\phi(k_1)$  and $J_{12}=\phi(2k_1)+2\phi(k_1)$ are very small if $k_1$ is large. Observe that   $\dfrac{\sin (t J_{1 1})}{t J_{1 1}} \sim 1$.
In the same manner, for $k_1$ large
\begin{equation}\label{x2int-n8}
\begin{split}
\int_0^tS(t-t')J_3dt' =& 2k_1k_2\psi(k_1-k_2) \, \dfrac{\sin J_{3 1}}{J_{3 1}}\sin \left( (k_1-k_2)x+J_{3 2} t  \right)\\
\sim & k_1k_2\psi(-1) \, \dfrac{t \sin t \phi(-1) }{ t\phi(-1)}\sin \left( x+J_{3 2} t   \right)\\
\sim &k_1^2 t \sin \left( x  +J_{3 2} t \right),
\end{split}
\end{equation}
where $J_{3 1}=\phi(k_1-k_2)-(\phi(k_1)-\phi(k_2)) $  and $J_{3 2}=\phi(k_1)-\phi(k_2)+\phi(k_1-k_2)$. 

Notice that for $n \in \Z$, one has $\mathcal{F}\{ \sin (nx+ \omega t)\}(k)=e^{i \omega k t}\mathcal{F}\{ \sin (nx)\}(k)$ and
\begin{equation}\label{l5kdvbbm-5}
\mathcal{F}\{ \sin (nx)\}(k)=\left\{
\begin{aligned}
0, \qquad &\text{if}\;k\neq n, -n,   \\
-\frac{i}2, \qquad &\text{if}\;k=n,   \\
\frac{i}2, \qquad &\text{if}\;k=-n.
\end{aligned}
\right. 
\end{equation}
Hence, for  $s<1$ and $k_1$ large, we can obtain
$$\|\mathcal{I}_4\|_{H^s}  \sim t |\tau(k_1+k_2)| (k_1+k_2)^s \leq t |\tau(k_1+k_2)| (k_1+k_2) \leq C t, $$ 
$$\|\mathcal{J}_3\|_{H^s} \sim k_1^2 t$$
and
$$\|\mathcal{J}_1\|_{H^s} \sim  tk_1^2 \psi(k_1) k_1^s \leq t k_1^2 \psi(k_1) k_1 \leq C t.$$
Similarly  it can be shown that for all $t\geq 0$
$$
\|\mathcal{J}_j\|_{H^s} \lesssim t k_1^2, \quad j=2,4 \qquad \textrm{and}  \qquad \|\mathcal{I}_j\|_{H^s} \lesssim t k_1^2, \quad j=1,2,3.
$$
Therefore,
\begin{equation}\label{p-n1}
\|\eta_1(\cdot, t)\|_{H^s} \sim k_1^{2(\sigma-1)+2}t = k_1^{2\sigma}t.
\end{equation}
From \eqref{p-n1}, we can infer that $H^s(\T)$-norm of $\eta_1$ can be made as big as we wish by choosing   $k_1$ sufficiently large.


Now, we move to estimate of $\zeta$.
For $T>0$, we denote by $X^T$ the space $C([0,T], H^1_{per})$. From \eqref{eq4.05}, Propositions \ref{multipliertau}, \ref{multiplierpsi} and  \ref{multiplierpsietax}, it follows that
\begin{equation}\label{xeq1.1}
\begin{split}
\|\zeta\|_{X^T}\lesssim &  T\|S(t)\eta_0\|_{X^T}(\|\eta_1\|_{X^T}+ \|\zeta\|_{X^T})+T(\|\eta_1\|_{X^T}+\|\zeta\|_{X^T})^2\\
&+T(\|\eta_1\|_{X^T}+\|\zeta\|_{X^T})^3+T(\|\eta_1\|_{X^T}+\|\zeta\|_{X^T})\|S(t)\eta_0\|_{X^T}^2+T\|S(t)\eta_0\|_{X^T}^3\\
&+T(\|\eta_1\|_{X^T}+\|\zeta\|_{X^T})^2\|S(t)\eta_0\|_{X^T}.
\end{split}
\end{equation}
Since $\|\eta_1\|_{X^T} \sim k_1^{2\sigma} T$ and  $\|S(t)\eta_0\|_{s} \sim k_1^{\sigma-1+s}$, with  $0< \sigma<1-s$ one has that $\|S(t)\eta_0\|_{X^T} \sim k_1^{\sigma}$. In this way, we can deduce that
\begin{equation}\label{xeq1.2}
\begin{split}
\|\zeta\|_{X^T}&\lesssim   T^2 k_1^{3\sigma}+ T k_1^{\sigma}\|\zeta\|_{X^T}+T^3 k_1^{4\sigma}+T \|\zeta\|_{X^T}^2\\
&\quad +T^4\ k_1^{6\sigma}+T \|\zeta\|_{X^T}^3+T^2  k_1^{4\sigma}+T  k_1^{2\sigma}\|\zeta\|_{X^T}+T k_1^{3\sigma}\\
&\quad +T^3 k_1^{5\sigma}+T  k_1^{\sigma}\|\zeta\|_{X^T}^2\\
&\lesssim  Tk_1^{3\sigma} (1+T+Tk_1^{\sigma}+T^2k_1^{\sigma}+T^2k_1^{2\sigma}+T^3k_1^{3\sigma})+Tk_1^{\sigma}(1+k_1^{\sigma})\|\zeta\|_{X^T}\\
&\quad + T(1+k_1^{\sigma})\|\zeta\|_{X^T}^2+ T\|\zeta\|_{X^T}^3\\
&=: \alpha+ \beta \mathcal{X}+\iota  \mathcal{X}^2+ T\mathcal{X}^3,
\end{split}
\end{equation}
where $\mathcal{X}=\mathcal{X}(T)=\|\zeta\|_{X^T}$. Let $\tilde{T}= k_1^{-\theta\sigma}$, with $ \theta>3$, then for all $T \leq \tilde{T}$ and $k_1$ large we get,
\begin{equation}\label{xeq1.3}
\begin{split}
\alpha \lesssim &k_1^{(3-\theta)\sigma}+k_1^{(3-2\theta)\sigma}+k_1^{(4-2\theta)\sigma}+k_1^{(4-3\theta)\sigma}+k_1^{(5-3\theta)\sigma}+ k_1^{(6-4\theta)\sigma}\lesssim  k_1^{(3-\theta)\sigma},\\
\beta \lesssim & k_1^{(1-\theta)\sigma}+k_1^{(2-\theta)\sigma} \lesssim k_1^{(2-\theta)\sigma},\\
\iota \lesssim & k_1^{-\theta\sigma}+k_1^{(1-\theta)\sigma} \lesssim k_1^{(1-\theta)\sigma}.
\end{split}
\end{equation}
Thus for all $T \leq \tilde{T}$, we have
\begin{equation}\label{xeq1.4}
\mathcal{X} \leq C ( k_1^{(3-\theta)\sigma}+k_1^{(2-\theta)\sigma} \mathcal{X}  + k_1^{(1-\theta)\sigma} \mathcal{X}^2+   k_1^{-\theta\sigma}\mathcal{X}^3),
\end{equation}
where $C \geq 1$ is a constant.   For all 
\begin{equation}\label{1xeq1.4}
k_1>(28C^3)^{1/\sigma}
\end{equation}
and for all $T\leq \tilde{T}$, one has that
\begin{equation}\label{xeq1.5}
\|\zeta\|_{X^T}=\mathcal{X}(T) \leq  2 C  k_1^{(3-\theta)\sigma}.
\end{equation}
This follows by a continuity argument. If possible, suppose that there exist $T' \in (0,\tilde{T}]$ such that $\mathcal{X}(T') >  2 C  k_1^{(3-\theta)\sigma}$.  Since $\mathcal{X}(0)=0$ and $\mathcal{X}(T)$ is a continuous function, by intermediate value theorem,  there exists $T^{*} \in (0,T')$ such that $\mathcal{X}(T^{*}) =  2 C  k_1^{(3-\theta)\sigma}$. Now, combining \eqref{xeq1.4} (with $T=T^{*}$ ) and \eqref{1xeq1.4} we arrive at
\begin{equation}\label{xeq1.6}
1 \leq 2 k_1^{(2-\theta)\sigma}C + 4 k_1^{(3-\theta)\sigma} k_1^{(1-\theta)\sigma}C^2+ 8  k_1^{-\theta\sigma} k_1^{2(3-\theta)\sigma}C^3 <14C^3k_1^{(2-\theta)\sigma}<\frac12,
\end{equation}
which is a contradiction. Therefore \eqref{xeq1.5} is true.  This concludes that the remainder $\zeta(x,t)$ is uniformly bounded in $H^1_{per}$ for $k_1>(28C^3)^{1/\sigma}$ and $t \leq \tilde{T}$. Now, we consider an increasing sequence $\{k_1^j\}, j=1, \cdots \infty$ such that
$$
\lim_{j \to \infty}k_1^j=\infty.
$$
Since $\sigma-1+s<0$, from  \eqref{norm-i1} we see that the initial data $\eta_0^j$ given in \eqref{eta_0} tends to zero in $H^s_{\textrm{per}}$. On the other hand \eqref{eta}, \eqref{p-n1} and \eqref{xeq1.5} imply that the solutions $\eta_j$ blow-up at times $T_j=(k_1^j)^{-\theta \sigma}$ and $T_j \to 0$ since $\theta$ and $\sigma$ are both positive. This concludes the proof of the theorem.
\end{proof}

\end{document}